\documentclass[graybox]{svmult}

\usepackage{mathptmx}       
\usepackage{helvet}         
\usepackage{courier}        
\usepackage{type1cm}        
\usepackage{makeidx}         
\usepackage{graphicx}        
\usepackage{multicol}        
\usepackage[bottom]{footmisc}
\usepackage{amsmath}
\usepackage{amssymb}

\makeindex             
\usepackage{algorithm}
\usepackage{algorithmic}

\usepackage{booktabs}
\newcommand{\ra}[1]{\renewcommand{\arraystretch}{#1}}

\usepackage[hidelinks]{hyperref}
\usepackage{graphicx}
\usepackage{enumerate}

\newtheorem{propos}{Proposition}

\newcommand{\LMe}{\varepsilon}
\newcommand{\LMo}{\omega}
\newcommand{\LMx}{\bar x}

\newcommand{\LMR}{\mathbb R}
\newcommand{\LME}{\mathbb E}

\newcommand{\LMn}[1]{\|#1 \|} 

\newcommand{\LMlr}[1]{\langle #1\rangle} 

\newcommand{\LMhx}{\hat x}
\newcommand{\LMhg}{\hat g}

\DeclareMathOperator{\LMprox}{prox}
\DeclareMathOperator{\alphargmin}{argmin}
\DeclareMathOperator{\LMid}{Id}
\DeclareMathOperator{\deltaist}{dist}

\DeclareMathOperator{\LMDiag}{Diag}

\smartqed 

\begin{document}

\title*{Block-coordinate primal-dual method for the nonsmooth minimization
over linear constraints}

\author{D. Russell Luke and Yura Malitsky}
\institute{D. Russell Luke \at Institute for Numerical and Applied
Mathematics,  University of G\"ottingen, 37083 G\"ottingen, Germany, \email{r.luke@math.uni-goettingen.de}
\and  Yura Malitsky \at Institute for Numerical and Applied
Mathematics,  University of G\"ottingen, 37083 G\"ottingen, Germany, \email{yurii.malitskyi@math.uni-goettingen.de}. }
\authorrunning{Block-coordinate primal-dual method}
%
%

\maketitle
\abstract*{
We consider the problem of minimizing a convex, separable,
nonsmooth function subject to linear constraints. The numerical method
we propose is a block-coordinate extension of the Chambolle-Pock
primal-dual algorithm.  We
prove convergence of the method without resorting to assumptions like
smoothness or strong convexity of the objective, full-rank condition
on the matrix, strong duality or even consistency of the linear
system.  Freedom from imposing the latter assumption permits
convergence guarantees for misspecified or noisy systems.
}

\abstract{
    We consider the problem of minimizing a convex, separable, nonsmooth
    function subject to linear constraints. The numerical method we propose is a block-coordinate
extension of the Chambolle-Pock primal-dual algorithm.  We prove convergence of the method
    without resorting to assumptions like smoothness or strong
    convexity of the objective, full-rank condition on the matrix,
    strong duality or even consistency of the linear system.  Freedom from
    imposing the latter assumption permits convergence guarantees for misspecified or
    noisy systems.
}
\bigskip



\section{Introduction}\label{sec:intro}

We propose a randomized coordinate primal-dual algorithm for
convex optimization problems of the form
\begin{equation}
    \label{eq:main2}
    \min_x g(x)\quad \text{s.t.}\quad x \in \alphargmin_z  \frac 1 2
    \LMn{Az-b}^2.
  \end{equation}
This is a generalization of the more commonly encountered linear constrained convex optimization problem
\begin{equation}
    \label{eq:main} \min_x g(x) \quad \text{s.t.}\quad Ax = b.
\end{equation}

When $b$ is in the range of $A$ problem \eqref{eq:main} and \eqref{eq:main2} have the same optimal solutions;
but \eqref{eq:main2} has the advantage of having solutions even when $b$ is {\em not} in the
range of $A$.  Such problems will be called {\em inconsistent} in what follows.
Of course, the solution set to \eqref{eq:main2} can be modeled by a problem with the format \eqref{eq:main} via the normal equations.
The main point of this note, however, is that the two models suggest very different algorithms with different
behaviors, as explained in \cite{malitsky2017} and in Section \ref{sec:applications} below.
%

%
We do not assume that
$g$ is smooth, but this is not our main concern.  Our main focus in this note
is the efficient use of problem structure.  In particular, we assume throughout that the problem can be
decomposed in the following manner.  For $x\in \LMR^n$, $A\in\LMR^{m\times n}$
\[ A  = [A_1,\dots, A_p]  \quad\mbox{ and }\quad
 g(x)  =\sum_{i=1}^p g_i(x_i),
\]
where $x_i\in\LMR^{n_i}$, $\sum_{i=1}^p n_i=n$,
$A_i \in \LMR^{m \times n_i}$ and
$g_i:\LMR^{n_i}\to (-\infty,+\infty]$ is proper, convex and lower
semi-continuous (lsc). The coordinate primal-dual method we propose
below allows one to achieve improved stepsize choice, tailored to the
individual blocks of coordinates.  To this, we add an intrinsic
randomization to the algorithm which is particularly well suited for
large-scale problems and distributed implementations. Another
interesting property of the proposed method is that in the absence of
the existence of Lagrange multipliers one can still obtain meaningful
convergence results.

Randomization is currently the leading technique for handling
extremely large-scale problems.  Algorithms employing some sort of
randomization have been around for more than 50 years, but they have
only recently emerged as the preferred -- indeed, sometimes the only
feasible -- numerical strategy for large-scale problems in machine
learning and data analysis.  The dominant randomization strategies can
be divided roughly into two categories. To the first category belong
stochastic methods, where in every iteration the full vector is
updated, but only a fraction of the given data is used. The main
motivation behind such methods is to generate descent directions
cheaply. The prevalent methods SAGA \cite{defazio2014saga} and SVR
\cite{johnson2013accelerating} belong to this group.  Another category
is coordinate-block methods. These methods update only one coordinate
(or one block of coordinates) in every iteration. As with stochastic
methods, the per iteration cost is very low since only a fraction of
the data is used, but coordinate-block methods can also be accelerated
by choosing larger step sizes.  Popular methods that belong to this
group are \cite{nesterov2012efficiency,fercoq2015accelerated}. A
particular class of coordinate methods is alternating minimization
methods, which appear to be a promising approach to solving nonconvex
problems, see
\cite{attouch2010proximal,bolte2014proximal,hesse2015proximal}.

To keep the presentation simple, we eschew many possible generalizations and extensions of our
proposed method.  For concreteness we focus our attention on the primal-dual algorithm (PDA)
of Chambolle-Pock
\cite{chambolle2011first}.
The PDA is a well-known first-order
method for solving saddle point problems with nonsmooth
structure.  It is based on the {\em proximal} mapping associated with a function $g$ defined by
$
\LMprox_{\tau g} = (\LMid + \tau \partial g)^{-1}
$
where $\partial g(\LMx)$ is the convex {\em subdifferential} of $g$ at $\LMx$,
defined as the set of all vectors
$v$ with
\begin{equation}\label{e:csd}
g(x)- g(\LMx) - \langle{v}, {x-\LMx}\rangle \geq 0\quad\forall x.
\end{equation}

The PDA applied to the Lagrangian of problem~\eqref{eq:main} generates two
sequences $(x^k), (y^k)$ via
\begin{equation}\label{eq:pda}
    \begin{aligned}
        x^{k+1} & = \LMprox_{\tau g}(x^k - \tau A^T y^{k})\\
        y^{k+1} & = y^k + \sigma (A(2x^{k+1}-x^k)-b).
    \end{aligned}
\end{equation}
Alternative approaches such as the alternating direction method of multipliers
\cite{Glowinski75} are also currently popular  for large-scale problems and
are based on the {\em augmented} Lagrangian of \eqref{eq:main}.
%
%
The advantage of PDA over ADMM, however, is that it does not require one to invert the matrix $A$, and hence can
be applied to very large-scale problems.
The PDA is very similar, in fact,  to a
special case of the proximal point algorithm~\cite{he-yuan:2012}
and of the so-called {\em proximal} ADMM \cite{shefi2014rate,eckstein1994some,banert2016fixing}.

\begin{algorithm}
  \begin{algorithmic}[1]
    \STATE Choose $x^0 \in \LMR^n$, $\sigma>0$, $\tau\in \LMR^p_{++}$  and set $y^0=u^0 =
    \sigma(Ax^0-b)$.
    \FOR{ $k=0,1,2,\dots$}
    \STATE $x^{k+1} = x^k$
    \STATE Pick an
    index $i\in \{1,\dots, p\}$ uniformly at random.
    \STATE $x^{k+1}_{i} =\LMprox_{\frac{\tau_i}{p} g_i}(x^k_i -
    \frac{\tau_i}{p} A^T_i y^k)$,\quad $t_i = x^{k+1}_i-x^k_i$
    \STATE $y^{k+1} = y^k + u^k + \sigma (p+1) A_it_i$
    \STATE $u^{k+1} = u^k + \sigma A_it_i$
\ENDFOR
\end{algorithmic}
\caption{Coordinate primal-dual algorithm}
\label{alg:pd}
\end{algorithm}
The procedure we study in this note is given in Algorithm \ref{alg:pd}.
The cost per iteration is very low: it
requires two dot products $A_it_i$, $A^T_iy^k$; and the full
vector-vector operation is needed only for the dual variables $y^k,
u^k$ in steps~$6$--$7$.  The algorithm will therefore be the most efficient if
$m \leq n$.  If in particular all blocks are of size $1$, that is $n_i
= 1$, $n=p$ and $A_i$ is just the $i$-th column of the matrix $A$,
then $A_it_i$ reduces to the vector-scalar multiplication and $A_i^T
y^k$ to the vector-vector dot product.  Moreover, if the matrix $A$ is
sparse and $A_i$ is its $i$-th column, then step~$7$ requires an
update of only those coordinates which are nonzero in $A_i$.  The
memory storage is also relatively small: we have to keep $x^k$, and two
dual variables $y^k$, $u^k$.  Another important feature of the
proposed algorithm is that it is well suited for the distributed optimization: 
since there are no operations with a full primal vector, we can keep
different blocks $x_i$ on different machines which are coupled
only over dual variables.

We want to highlight that with $p=1$ the proposed algorithm indeed
coincides with the primal-dual algorithm of
Chambolle-Pock~\cite{chambolle2011first}. In fact, in this case it is
not difficult to prove by induction that $u^k = \sigma (Ax^k-b)$ and
hence, $y^{k+1} = y^k + \sigma (A(2x^{k+1}-x^k)-b)$.

To the best of our knowledge the first randomized extension of Chambolle-Pock algorithm can be found in
\cite{zhang2015stochastic}. This method was proposed for
solving the empirical risk minimization problem and is very popular
for such kinds of problems. However, it converges under quite
restrictive assumptions that never hold for
problem~\eqref{eq:main}, see more details in
Section~\ref{sec:applications}. Another interesting coordinate
extension of the primal-dual algorithm is studied in
\cite{fercoq2015coordinate}.  This does not require any special
assumptions like smoothness or strong convexity, but, unfortunately,
it requires an increase in the dimensionality of the problem, which for our
targeted problems is counterproductive.

Recently there have also appeared coordinate methods for abstract
fixed point algorithms
\cite{combettes2015stochastic,iutzeler2013asynchronous}. Although they
provide a useful and general way for developing coordinate algorithms,
they do not allow the use of larger stepsizes --- one of the major
advantages of coordinate methods.

There are also some randomized methods for a particular choice of $g$
in \eqref{eq:main}. For $g\equiv 0$
paper~\cite{strohmer2009randomized} proposes a stochastic version of
the Kaczmarz method, see also \cite{leventhal2010randomized} and a
nice review~\cite{wright2015coordinate}. When $g$ is the elastic net
regularization (sum of $l_1$ and squared $l_2$ norms), papers
\cite{lorenz2014linearized,schopfer2016linear} provide a stochastic
sparse Kaczmarz method. Those methods belong to the first category of
randomized methods by our classification, i.e., in every iteration
they require an update of the whole vector.

In \cite{malitsky2017} the connection between the
primal-dual algorithm \eqref{eq:pda} and the Tseng proximal gradient
method~\cite{tseng08} was shown. On the other hand,
paper~\cite{fercoq2015accelerated} provides a coordinate extension of
the latter method for a composite minimization problem. Based on these
two results we propose a coordinate primal-dual method which is an
extension of the original PDA. The key feature of Algorithm \ref{alg:pd}
is that it requires neither strong duality
nor the consistency of the system $Ax=b$ to achieve good numerical 
performance with convergence guarantees. This allows us, for
instance, to recover the signals from noisy measurements without
any knowledge about the noise or without the need to tune some external 
parameters as one must for lasso or basis denoising
problems~\cite{chen2001atomic,candes2006stable}.


\bigskip

In the next section we provide possible applications and connections
to other approaches in the literature.  Section~\ref{sec:analysis} is
dedicated to the convergence analysis of our method. We also give an
alternative form of the method and briefly discuss possible
generalizations of the method. Finally, Section~\ref{sec:numer-exper}
details several numerical examples.



\section{Applications}
\label{sec:applications}

We briefly mention a few of the more prominent applications for Algorithm \ref{alg:pd}.
We begin with the simplest of these.

\textbf{Linear programming. } The linear programming
problem
\begin{equation}
  \label{eq:lp} \min_x\, \LMlr{c,~x}\quad \text{s.t. }\, Ax=b, \quad
x\geq 0
\end{equation}
is a special case of \eqref{eq:main} with $g(x) =
\LMlr{c,~x}+\delta _{x\geq 0}(x)$. In this case $g$ is fully separable. As a
practical example, one can consider the optimal transport
problem~(see for instance \cite{santambrogio2015optimal}).

\bigskip
\textbf{Composite minimization.}
The composite minimization problem
\begin{equation}
    \label{eq:comp} \min_v f(Kv) + r(v),
\end{equation} where both $f, r$ are convex functions, $v\in \LMR^q$,
and $K\in \LMR^{m\times q}$ is a linear operator. For simplicity assume
that both $f$ and $r$ are nonsmooth but allow for easy evaluation of the associated
proximal mappings. In this case
the PDA in \cite{chambolle2011first}
is a widely used method to solve such problems. It is
applied to~\eqref{eq:comp} written in the primal-dual form
\begin{equation}
    \label{eq:comp_pd} \min_v\max_w r(v) + \langle Kv,w\rangle-f^*(w).
\end{equation} Alternatively, one can recast~\eqref{eq:comp} in the
form \eqref{eq:main}.  To see this, let $x = (v, w)$, $g(x) = r(v) +
f(w)$, and $Ax = Kv-w$. Then \eqref{eq:comp} is equivalent to
\begin{equation}
    \label{eq:comp2} \min_x g(x) \quad \text{s.t.}\quad Ax = 0.
\end{equation} Such reformulation is typical for the augmented
Lagrangian method or ADMM~\cite{boyd2011distributed}.  However, this
is not very common to do for the primal-dual method, since
problem~\eqref{eq:comp_pd} is already simpler
than~\eqref{eq:comp2}.
Although the number of matrix-vector operations
for both applications of the PDA remains the same, in \eqref{eq:comp2} we have a larger
problem: $x\in \LMR^{m+q}$, and the multiplier $y\in \LMR^m$ instead of $v\in \LMR^q$, $w\in
\LMR^m$.
However, the advantage for formulation \eqref{eq:comp2} over
\eqref{eq:comp_pd} is that the updates of the
most expensive variable
$x=(v,w)$ can be done in parallel, in contrast to the sequential update
of $v$ and $w$ when we apply PDA to~\eqref{eq:comp_pd}.


Still, the main objection to~\eqref{eq:comp2} is that PDA treats the
matrix $A$ as a black box -- it does not exploit its structure. In
particular, we have $\lambda(A) = \lambda(K)+1$ which globally defines the
stepsizes.  But in our case
$A = [\begin{array}{c|c} K & -I \end{array}]$, whose structure is very tempting to
use: for the block $w$ in $x=(v,w)$ we would like to apply larger
steps that come from $\lambda(I) = 1 \leq \lambda(K)$ (the last inequality is
typical). Fortunately, the proposed algorithm does exactly this: for each block
the steps are defined only by the respective
block of $A$.  This is very similar in spirit to the proximal heterogeneous implicit-explicit
method studied in \cite{hesse2015proximal}.  Notice that the paper~\cite{pock2011diagonal} provides
only the possibility to use different weights for different blocks but
it does not allow one to enlarge stepsizes. In ~\cite{malitsky2016first} a
linesearch strategy is introduced that in fact allows one to use larger
steps, but still this enlargement is based on the inequalities for the
whole matrix $A$ and the vector $x$.  Fercoq and
Bianchi~\cite{fercoq2015coordinate} propose an extension (coordinate)
of the primal-dual method that takes into account the structure of the
linear operator, but this modification requires to increase the dimension
of the problem.


Algorithm~\ref{alg:pd} can be applied without any
smoothing of the nonsmooth functions $f$ and
$r$. This is different from the approach of
\cite{zhang2015stochastic,palaniappan2016stochastic}, where the
convergence is only shown when the fidelity term $f$ is smooth and
the regularizer $r$ is strongly convex.

\bigskip \textbf{Distributed optimization.} The aim of distributed
optimization over a network is to minimize a global objective using
only local computation and communication. Problems in contemporary
statistics and machine learning might be so large that even storing
the input data is impossible in one machine. For applications and
more in-depth overview we refer the reader to
\cite{boyd2011distributed,duchi2012dual,Bertsekas:1989,nedic2010constrained}
and the references therein.

By construction, Algorithm~\ref{alg:pd} is
distributed and hence ideally suited for these types of problems. Indeed, we can
assign to the $i$-th node of our network the block of variables
$x^i$ and the respective block matrix  $A_i$.  The dual variables $y$ and $u$
reside on the central processor. In each
iteration of the algorithm implemented with this architecture only one random
node is activated, requiring communication of the current dual vector $y^k$ from
the central processor in order to compute the update to the block $x_i$.  The node then
returns $A_it_i$ to the central processor.

The model problem~\eqref{eq:main} is also particularly
well-suited for a distributed optimization. Suppose we want to
minimize a convex function $f\colon \LMR^m\to (-\infty, +\infty]$. A
common approach is to assign for each node a copy of $x$ and solve a
constrained problem (the {\em product space formulation}):
\begin{equation}
    \label{eq:distr} \min_{x_1,\dots, x_p}f(x_1)+\dots + f(x_p)\quad
\text{s.t.}\quad x_1=\dots = x_p.
\end{equation} In fact, instead of $x_1= \dots =x_p$ we can introduce
a more general but equivalent\footnote{This means $Ax=0$ if and only
if $x_1=\dots = x_p$.} constraint $Ax = 0$, where $x=(x_1,\dots, x_p)$
and the matrix $A$ describes the given network. By this, we arrive at
the following problem
\begin{equation}
    \label{eq:distr2} \min_x g(x) \quad \text{s.t.}\quad Ax = 0,
\end{equation} where $g(x) = f(x_1) + \dots + f(x_p)$ is a separable
convex function. Solving such problems in a stochastically, 
where in every iteration only one or few nodes are activated,
has attracted a lot of attention recently,
see~\cite{latafat2017new,fercoq2015coordinate,bianchi2014stochastic}.

\bigskip \textbf{Inverse problems.}  Linear systems remain a central
occupation of inverse problems.  The
interesting problems are {\em ill-posed} and the
observed data $A$, $b$ is noisy, requiring an appropriate {\em regularization}.
A standard way is to consider either\footnote{The left and
right problems are also known as Tikhonov and Morozov regularization
respectively.}
\begin{equation}
    \label{eq:inverse} \min_x g(x) + \frac{\delta }{2}\LMn{Ax-b}^2 \qquad
\text{or}\qquad \min_xg(x) \quad \text{s.t.}\quad \LMn{Ax-b}\leq \delta ,
\end{equation}
where $g$ is the regularizer which describes our a
priori knowledge about a solution and $\delta >0$ is some
parameter. The issue of how to choose this parameter is a
major concern of numerical inverse problems. For the right hand-side problem this is
easier to do: usually $\delta $ corresponds to the noise level of the given
problem, but the optimization problem itself is harder than the left
one due to the nonlinear constraints.  Nevertheless, it can be still
efficiently solved via PDA.  Let $y = Ax-b$ and $h$ be the indicator
function of the closed ball $B(0, \delta)$. Then the above problem can
be expressed as
\begin{equation}
    \label{eq:inverse2} \min_{x,y} g(x) + h(y)\quad \text{s.t.}\quad
Ax-y = b,
\end{equation} which is a particular case of \eqref{eq:main}. Again
the coordinate primal-dual method, in contrast to the original PDA,
has the advantage of exploiting the structure of the matrix
$[\begin{array}{c|c} A & -I \end{array}]$ which allows one to use larger
steps for faster convergence of the method.

There is another approach which we want to discuss. In many
applications we do not need to solve a regularized problem exactly, the
so-called early stopping may even help to obtain a better
solution. Theoretically, one can consider another regularized problem:
$ \min_x g(x)$ such that $Ax=b$.  Since during iteration we can easily
control the feasibility gap $\LMn{Ax^k-b}$, we do not need to converge
to a feasible solution $x^*$, where $Ax^*= b$, but stop somewhere
earlier. The only issue with such approach is that the system $Ax=b$
might be inconsistent (and due to the noise this is often the
case). Hence to be precise, we have to solve the following problem
\begin{equation}
    \label{eq:inverse_pr} \min_x g(x) \quad \text{s.t.}\quad x \in
\alphargmin_z \bigl\{f(z):=\frac 1 2 \LMn{Az-b}^2\bigr\}.
\end{equation} Obviously, the above constraint is equivalent to the
$A^TAx = A^Tb$.  Fortunately, we are able to show that our proposed
method (without any modification) in fact solves
\eqref{eq:inverse_pr}, so it does not need to work with $A^TA$ that
standard analysis of PDA requires. Notice that $A^TA$ is likely to be
less sparse and more ill-conditioned than $A$.  \bigskip

\section{Analysis}\label{sec:analysis}
We first introduce some notation. For any vector
$\omega = (\omega_1,\dots, \omega_p)\in \LMR^p_{++}$ we define the weighted norm by
$\LMn{x}_\omega^2 := \sum \omega_i \LMn{x_i}^2$ and the weighted proximal
operator $\LMprox_{g}^\omega$ by
\[
\LMprox_{g}^\omega := (\LMid + \LMDiag(\omega^{-1})\circ \partial g)^{-1} = (\LMDiag(\omega)
+ \partial g)^{-1}\circ \LMDiag(\omega).
\]
The weighted proximal operator has
the following characteristic property (prox-inequality):
\begin{equation}\label{prox_charact} \LMx= \LMprox_{g}^\omega z \quad
\Leftrightarrow \quad \LMlr{\LMDiag(\omega)(\LMx- z), ~x- \LMx} \geq g(\LMx) - g(x)
\quad \forall x\in \LMR^n.
\end{equation}
From this point onwards we will fix
\begin{equation}\label{eq:feas resid}
f(x) := \frac 1 2 \LMn{Ax-b}^2.
\end{equation}
Then $\nabla
f(x) = A^T(Ax-b)$ and its partial derivative corresponding to $i$-th
block is $\nabla_i f(x) = A_i^T(Ax-b)$. Let $\lambda = (\lambda_1, \dots,
\lambda_p)$, where $\lambda_i$ is the largest
eigenvalue of $A_i^T A_i$, that is $\lambda_i = \lambda_{\max}(A_i^TA_i)$. Then
the Lipschitz constant of the $i$-th partial gradient is $\LMn{A_i}^2 =
\lambda_i$.  By $U_i\colon \LMR^n \to \LMR^{n}$ we denote the projection
operator: $U_i x = (0,\dots, x_i, \dots, 0)$.
Since $f$ is quadratic, it follows that for any $x,y \in \LMR^n$
\begin{equation}
    \label{eq:quadr}
    f(y) - f(x) - \LMlr{\nabla f(x), y-x} = \frac 1 2 \LMn{A(x-y)}^2.
\end{equation}
We also have that for any $x, t\in \LMR^n$
\begin{equation}
    \label{eq:lipsch} f(x + U_i t) = f(x) + \LMlr{\nabla f(x), ~U_i t} +
\frac 1 2 \LMn{A_it_i}^2\leq f(x) + \LMlr{\nabla f(x), ~U_i t} +
\frac{\lambda_i}{2} \LMn{t_i}^2.
\end{equation}
Now since $i$ is a uniformly distributed random number
over $\{1,2,\dots,p\}$, from the above it follows
\begin{equation}
    \label{eq:exp} \LME[f(x+U_i t)]\leq f(x) + \frac{1}{p} \LMlr{\nabla f(x), ~t} +
\frac{1}{2p} \LMn{t}^2_{\lambda}.
\end{equation}

For our analysis it is more convenient to work with
Algorithm~\ref{alg:tseng} given below. It is equivalent to
Algorithm~\ref{alg:pd} when the
random variable $i$ for the blocks are the same, though this
might be not obvious at first
glance. We give a short proof of this fact. Notice also
that Algorithm~\ref{alg:tseng} is entirely primal. The
formulation of Algorithm~\ref{alg:tseng} with $p=1$ (not random!) is  related to the
Tseng proximal gradient method \cite{malitsky2017,tseng08} and
to its stochastic extension, the APPROX method \cite{fercoq2015accelerated}.
The proposed method requires stepsizes: $\sigma >0$ and
$\tau \in \LMR^p_{++}$. The necessary condition for convergence of
Algorithms~\ref{alg:pd} and \ref{alg:tseng}
is, as we will see, $\tau_i \sigma \LMn{A_i}^2 <
1$. We have strict inequality for the same reason that one needs
$\tau \sigma \LMn{A}^2 < 1$ in the original PDA.
\begin{algorithm}
  \begin{algorithmic}[1]
    \STATE Choose $x^0 \in \LMR^n$, $\sigma >0$,  $\tau\in \LMR^p_{++}$ and set
    $s^0=x^0$, $\theta_0=1$.
    \FOR{ $k=0,1,2,\dots$}
    \STATE $z^k= \theta_k x^k+ (1-\theta_k) s^k$
    \STATE $x^{k+1} = x^k$, $ s^{k+1}=z^k$
    \STATE Pick an index $i\in \{1,\dots, p\}$ uniformly at
    random.
    \STATE $x^{k+1}_{i} = \LMprox_{\frac{\tau_i}{p} g_i}(x^k_i -
    \frac{\tau_i \sigma}{p \theta_k} \nabla_i f(z^k))$
    \STATE $s^{k+1}_{i} =z^k_i+{p}{\theta_k}(x^{k+1}_i-x^k_i)$
    \STATE $\theta_{k+1} = \frac{1}{k+2}$
\ENDFOR
\end{algorithmic}
\caption{Coordinate primal-dual algorithm (equivalent form)}
\label{alg:tseng}
\end{algorithm}

Let $f_* = \min f$ and $S$ be the solution set of
\eqref{eq:main2}. Observe that if the linear system $Ax=b$ is
consistent, then $f_* = 0$ and $\alphargmin f = \{x\colon
Ax=b\}$. Otherwise, $f_*>0$ and $\alphargmin f = \{x\colon A^TAx=A^Tb\}$.
We will often use an important simple identity:
\begin{equation}
    \label{eq:id_f}
    f(x) - f(\LMx) = \frac 1 2 \LMn{A(x-\LMx)}^2 \qquad (\forall x\in \LMR^n)
    (\forall \LMx\in S).
\end{equation}

\begin{propos}
    If the index $i$ selected at iteration $k$ in step $4$ of Algorithm~\ref{alg:pd} is identical to the index $i$ selected at iteration $k$ in step $5$ of
    Algorithm \ref{alg:tseng}, then both algorithms with the same
    starting point $x^0$ generate the same sequence $(x^k)$.
\end{propos}

\begin{proof} We show that from Algorithm~\ref{alg:tseng} one can
recover all iterates of Algorithm~\ref{alg:pd} by setting $y^k =
\frac{\sigma}{\theta_k}(Az^k-b)$, $u^k = \sigma (Ax^k-b)$. Then the proposition
follows, since with $\nabla_i f(x) = A_i^T(Ax-b)$ we
have
\begin{equation*}
    x^{k+1}_i = \LMprox_{\frac{\tau_i}{p} g_i}(x^k_i - \frac{\tau_i
    \sigma}{p \theta_k} A^T_i(Az^k-b)) = \LMprox_{\frac{\tau_i}{p} g_i}(x^k_i - \frac{\tau_i}{p}
    A^T_iy^k).
\end{equation*}
Evidently, for $k = 0$, one has $y^0 = \sigma (Az^0-b) =
\sigma (Ax^0-b) = u^0$. Assume it holds for some $k\geq 0$.  By step~3 in
Algorithm~\ref{alg:tseng}, we have
\begin{equation*}
    Az^{k+1} = \theta_{k+1} Ax^{k+1} +
    (1-\theta_{k+1})As^{k+1} = \theta_{k+1}(Ax^{k} + A_it_i) +
    (1-\theta_{k+1})(Az^k + p\theta_kA_i t_i),
\end{equation*}
where we have used that $Ax^{k+1} = Ax^k + A_it_i$. Hence,
\begin{equation*}
    \frac{\sigma}{\theta_{k+1}}(Az^{k+1}-b) = \sigma(Ax^k-b) +
    \frac{\sigma}{\theta_k}(Az^k-b) + \sigma(p+1)A_i t_i = u^k + y^k + \sigma (p+1)A_i
    t_i,
\end{equation*}
thus $y^{k+1} = \frac{\sigma}{\theta_{k+1}}(Az^{k+1}-b)$. Finally,
$\sigma(Ax^{k+1}-b) = u^k + \sigma A_it_i = u^{k+1}$.
\qed
\end{proof}

\bigskip

We are now ready to state our main result. Since the iterates given by
Algorithm~\ref{alg:tseng} are random variables, our convergence result
is of a probabilistic nature. Notice also that equalities and
inequalities involving random variables should be always understood to
hold almost surely, even if the latter is not explicitly stated.
\begin{theorem}\label{th:main}
    Let $(x^k), (s^k)$ be given by Algorithm~\ref{alg:tseng},
    $\tau_i \sigma \LMn{A_i}^2 <1$ for all $i=1,\dots, p$, and $S$ be the
    solution set of \eqref{eq:main2}. Then
\begin{enumerate}[(i)]
    \item\label{th:main i} If there exists a Lagrange multiplier for problem
    \eqref{eq:main2}, then $(x^k)$ and $(s^k)$ converge a.s.\ to a
    solution of \eqref{eq:main2} and 
    $f(x^k)-f_* = o(1/k)$, $f(s^k)-f_* = O(1/k^2)$ a.s.
for the feasibility residual \eqref{eq:feas resid}.
  \item\label{th:main ii} If $S$ is a bounded set and $g$ is bounded
    from below, then almost surely all limit points of $(s^k)$ belong
    to $S$ and $f(s^k)-f_* = o(1/k)$.
    \end{enumerate}
\end{theorem}

The proof of Theorem  \ref{th:main} is based on several simple lemmas which we
establish first.
The next lemma uses the following notation for the
full update, $\LMhx^{k+1}$,  of $x^k$ in the $k$-th iteration of
Algorithm~\ref{alg:tseng}, namely
\begin{equation}\label{eq:xhat} \LMhx^{k+1} = \LMprox^{\tau^{-1}}_{g/p}
\left(x^k - \frac{k+1}{p} \sigma
\LMDiag(\tau) \nabla f(z^k)\right).
\end{equation}

\begin{lemma}\label{prox-ineq}
 For any fixed $x\in\mathbb{R}^n$ and any $k\in\mathbb{N}$ the following 
inequality holds:
        \begin{multline}
        \label{eq:3} \frac{p}{2}\LMn{\LMhx^{k+1}-x^k}^2_{\tau^{-1}} +
\frac{\sigma}{\theta_k} \LMlr{\nabla f(z^k),~ \LMhx^{k+1}-x}+
g(\LMhx^{k+1})-g(x) \\ \leq \frac{p}{2}\LMn{x^{k}-x}^2_{\tau^{-1}} -
\frac{p}{2}\LMn{\LMhx^{k+1}-x}^2_{\tau^{-1}}.
\end{multline}
\end{lemma}
    \begin{proof} By the prox-inequality~\eqref{prox_charact} with
    $\LMhx^{k+1}$ given by \eqref{eq:xhat}
    we have
\begin{equation}
        \label{eq:2} \LMlr{\LMDiag(\tau^{-1})(\LMhx^{k+1}-x^k), ~x-\LMhx^{k+1}}
+ \frac{\sigma}{p\theta_k} \LMlr{\nabla f(z^k),~x-\LMhx^{k+1}} \geq
\frac{1}{p}(g(\LMhx^{k+1}) -g(x)).
\end{equation}
The statement then follows from the identity
\[
 \LMlr{\LMDiag(\tau^{-1})(\LMhx^{k+1}-x^k), ~x-\LMhx^{k+1}} = \frac{1}{2}\LMn{x^{k}-x}^2_{\tau^{-1}} -
\frac{1}{2}\LMn{\LMhx^{k+1}-x}^2_{\tau^{-1}} - \frac{1}{2}\LMn{\LMhx^{k+1}-x^k}^2_{\tau^{-1}}.
\]
\qed
\end{proof}

 The next result provides a bound on the expectation of the residual at the $(k+1)$th iterate, conditioned on the $k$th iterate,
 which we denote by $\LME_k$.
\begin{lemma}\label{l:exp}
For any $x^*\in S$ 
\begin{multline}
    \label{eq:1} \frac{1}{\theta_{k}^2}\LME_k [f(s^{k+1})-f_*] \leq
\frac{1}{\theta_{k-1}^2}(f(s^k)-f_*) - (f(x^k)-f_*)+ \frac{1}{\theta_k}
\LMlr{\nabla f(z^k), ~\LMhx^{k+1}-x^*} \\ + \frac{p}{2}
\LMn{\LMhx^{k+1}-x^k}^2_\lambda.
    \end{multline}
\end{lemma}
\begin{proof} First, by \eqref{eq:exp}
    \begin{equation}
        \label{eq:est_exp} \LME_k[f(s^{k+1})] \leq f(z^k) + \theta_k
\LMlr{\nabla f(z^k),~ \LMhx^{k+1}-x^k} + \theta_k^2 \frac{p}{2}
\LMn{\LMhx^{k+1}-x^k}^2_\lambda.
    \end{equation} Since $f$ is quadratic \eqref{eq:feas resid}, by
    \eqref{eq:quadr} we have
 \begin{multline}\label{bi2:3} f(s^k)-f(z^k) = \LMlr{\nabla
f(z^k),s^k-z^k} + \frac 1 2 \LMn{A(s^k-z^k)}^2\\ = \frac{\theta_k}{1-\theta_k}
\LMlr{\nabla f(z^k),z^k-x^k} + \frac{1}{2} \LMn{A(s^k-z^k)}^2.
\end{multline}
By \eqref{eq:quadr} and  $\|\alpha a + (1-\alpha)b\|^2 = \alpha \|a\|^2 +
(1-\alpha)\|b\|^2 - \alpha(1-\alpha)\|a-b\|^2$, we have
 \begin{multline}\label{bi2:4} f_* - f(z^k) = f(x^*)-f(z^k) = \LMlr{\nabla f(z^k),
x^*-z^k} + \frac 1 2 \LMn{A(z^k-x^*)}^2 \\ = \LMlr{\nabla f(z^k), x^*-z^k}
+ \theta_k (f(x^k)-f_*) + (1-\theta_k)(f(s^k) -f_*) -
\frac 1 2 \theta_k(1-\theta_k)\LMn{A(x^k-s^k)}^2.
\end{multline}
 Notice that $x^k-s^k =
\frac{1}{\theta_k}(s^k-z^k)$. Hence, summation of
$\frac{1-\theta_k}{\theta_k}\eqref{bi2:3}$ and \eqref{bi2:4} yields
\begin{equation}
    \label{eq:4} \frac{1-\theta_k}{\theta_k} f(s^k) - \frac{1}{\theta_k} f(z^k)
+ f_* = \LMlr{\nabla f(z^k),x^*-x^k} + \theta_k (f(x^k)-f_*) +
(1-\theta_k)(f(s^k)-f_*),
\end{equation} from which we conclude
\begin{equation}
    \label{eq:5} \frac{1}{\theta_k^2} f(z^k) + \frac 1 \theta_k \LMlr{\nabla
f(z^k),x^*-x^k} =
\frac{(1-\theta_k)^2}{\theta_k^2}f(s^k)+\frac{2-\theta_k}{\theta_k}f_* -
(f(x^k)-f_*).
\end{equation}
Now summing up \eqref{eq:est_exp}
multiplied by $\frac{1}{\theta_k^2}$ and \eqref{eq:5}, and using the
identity $\frac{1-\theta_k}{\theta_k}=\frac{1}{\theta_{k-1}}$, we obtain
\begin{multline}
    \label{eq:6} \frac{1}{\theta_{k}^2}\LME_k[f(s^{k+1})-f_*] \leq
\frac{1}{\theta_{k-1}^2}(f(s^k)-f_*) - (f(x^k)-f_*)+ \frac{1}{\theta_k}
\LMlr{\nabla f(z^k), \LMhx^{k+1}-x^*} \\ + \frac{p}{2}
\LMn{\LMhx^{k+1}-x^k}^2_\lambda.
\end{multline}
\qed
\end{proof}

The conditional expectations have the following useful representations.
Expanding the expectation
    \begin{multline*}
        \LME_k [\LMn{x^{k+1}-x}^2_{\tau^{-1}}] = \sum_{i=1}^p\tau_i^{-1} \LME_k[\LMn{x^{k+1}_i-x_i}^2] \\
= \sum_{i=1}^p \tau_i^{-1}\bigl(\frac 1 p \LMn{\LMhx^{k+1}_i-x}^2+
\frac{p-1}{p} \LMn{\LMhx^{k}_i-x}^2\bigr) = \frac{1}{p} \LMn{\LMhx^{k+1}-x}^2_{\tau^{-1}} +
\frac{p-1}{p} \LMn{x^{k}-x}^2_{\tau^{-1}}.
    \end{multline*}
This yields
\begin{equation}\label{eq:xk}
\LME_k\Bigl[\frac{p^2}{2}\LMn{x^{k+1}-x}^2_{\tau^{-1}}\Bigr] =
\frac{p}{2}\LMn{\LMhx^{k+1}-x}^2_{\tau^{-1}} +
\frac{p^2-p}{2}\LMn{x^k-x}^2_{\tau^{-1}}.
\end{equation}
Another characterization we use follows similarly, namely
    \begin{equation*}
        \LME_k[g(x^{k+1})] = \sum_{i=1}^p\bigl( \frac{1}{p}
g_i(\LMhx_i^{k+1}) + \frac{p-1}{p}g_i(x^k_i)\bigr) = \frac{1}{p}
g(\LMhx^{k+1})+\frac{p-1}{p} g(x^k)
\end{equation*}
which gives
\begin{equation}\label{eq:g}
\LME_k [g(x^{k+1})] = \frac{1}{p}
g(\LMhx^{k+1}) + \left(1-\frac{1}{p}\right)g(x^k).
\end{equation}

The next technical lemma is the last of the necessary preparations before proceding to
the proof of the main theorem.
\begin{lemma}
    \label{l:sk} The identity $s^{k+1} = \sum_{j=0}^{k+1} \beta^j_{k+1} x^{j}$ holds where
the coefficients $(\beta_{k+1}^j)_{j=0}^{k+1}$ are nonnegative and sum to $1$.
In particular,
    \begin{equation}
        \label{eq:recurs} \beta_{k+1}^j =
        \begin{cases} (1-\theta_k)\beta^j_{k}, & j = 0,\dots, k-1,\\
p\theta_{k-1}(1-\theta_k) -(p-1)\theta_k, & j =k,\\ p\theta_k, & j = k+1.
        \end{cases}
    \end{equation} and $\beta_{k+1}^k + (p-1)\theta_k = (1-\theta_k)\beta_{k}^k$.
\end{lemma}
\begin{proof} It is easy to prove by induction.  For the reference see
    Lemma~2 in \cite{fercoq2015accelerated}.
    \qed
\end{proof}

The proof of Theorem \ref{th:main} consists of three parts. The first contains an important
estimation, derived from previous lemmas. In the second and third
parts we respectively and independently prove \eqref{th:main i} and
\eqref{th:main ii}.  The condition in \eqref{th:main ii} is not
restrictive and often can be checked in advance without any
effort. In contrast, verifying strong duality (existence of Lagrange multipliers) is usually very difficult.

\begin{proof}[Theorem \ref{th:main}] Since $\tau_i \sigma \LMn{A_i}^2< 1$ for all $i=1,\dots,p$,
there exists $\LMe > 0 $ such that $\tau_i^{-1}-\sigma \LMn{A_i}^2 \geq
\LMe$. This yields for all $x \in \LMR^n$, $\LMn{x}_{\tau^{-1}}^2- \sigma
\LMn{x}_{\lambda}^2 \geq \LMn{x}_\LMe^2$.

By convexity of $g$,
\begin{equation*}
g(s^k) = g(\sum_{j=0}^k \beta_k^j x^j) \leq  \sum_{j=0}^k \beta_k^j  g(x^j) =: \LMhg_k.
\end{equation*}
Let  $\hat F_k:= \LMhg_k + \sigma k (f(s^k)-f_*) \geq \LMhg_k + \sigma k (f(s^k)-f_*)$.

By Lemma~\ref{l:sk} and \eqref{eq:g}, it follows
    \begin{align} \LME_k[\hat g_{k+1}] & =
\LME_k[\beta_{k+1}^{k+1}g(\LMhx^{k+1})] +\sum_{j=0}^k \beta_{k+1}^j g(x^j) \\ &
=\theta_k g(\LMhx^{k+1}) + ((p-1)\theta_k+ \beta_{k+1}^k)g(x^k) +
\sum_{j=0}^{k-1}\beta_{k+1}^j g(x^j) \\ & = \theta_k g(\LMhx^{k+1}) +
(1-\theta_k)\beta_{k}^k g(x^k) + (1-\theta_k)\sum_{j=0}^{k-1}\beta_k^jg(x^j)\\ & =\theta_k
g(\LMhx^{k+1}) +(1-\theta_k)\sum_{j=0}^{k}\beta_k^jg(x^j) =  \theta_k
g(\LMhx^{k+1}) + (1-\theta_k)\LMhg_k. \label{eq:gk}
    \end{align}
Let $x^*\in S$. Setting $x=x^*$ in \eqref{eq:3} and adding to $\sigma$ times \eqref{eq:1} yields
\begin{align}
  \frac{\sigma}{\theta_{k}^2}\LME_k[f(s^{k+1})-f_*] \leq &
 \frac{\sigma}{\theta_{k-1}^2}(f(s^k)-f_*) - \sigma (f(x^k)-f_*) + \frac{p}{2}\LMn{x^{k}-x^*}^2_{\tau^{-1}} \\ &
\!\!\!\!\!\!\!\!\!\! - \frac{p}{2}\LMn{\LMhx^{k+1}-x^*}^2_{\tau^{-1}} - \frac{p}{2} \LMn{\LMhx^{k+1}-x^k}^2_\LMe - (g(\LMhx^{k+1})-g_*). \label{eq:fk}
\end{align}
Summing up \eqref{eq:xk} with $x=x^*$, \eqref{eq:gk} multiplied by $1/\theta_k = (k+1)$, and \eqref{eq:fk}, we obtain
\begin{multline}
    \label{eq:10}
    \LME_k\Bigl[\frac{p^2}{2}\LMn{x^{k+1}-x^*}^2_{\tau^{-1}} + (k+1)(\hat F_{k+1} - g_*)\Bigr] \leq
\frac{p^2}{2}\LMn{x^{k}-x^*}^2_{\tau^{-1}} + k (\hat F_{k}-g_*) \\ -
\sigma (f(x^k)-f_*)- \frac{p}{2} \LMn{\hat x^{k+1}-x^k}^2_\LMe.
\end{multline}
If the term inside the expectation in \eqref{eq:10},
$V_k(x^*):=\frac{p^2}{2}\LMn{x^{k}-x^*}^2_{\tau^{-1}} + k (\hat
F_{k}-g_*)$, were nonnegative, then one could apply the supermartingale theorem~\cite{Bertsekas:1989} to obtain almost sure
convergence directly.
In our case this term need not be nonnegative, however it suffices to have $V_k(x^*)$
bounded from below.  With the assumption that
there exists a Lagrange multiplier $u^*$, we shall show this.

\eqref{th:main i} Let $u^*$ be a Lagrange multiplier for a solution $x^*$ to problem
\eqref{eq:main2}. Then by definition of the saddle point we have
\begin{eqnarray}
    g(s^k)- g_* = g(s^k)-g(x^*) &\geq&
    \LMlr{A^TA u^*, s^k-x^*}\nonumber\\
    &\geq& -\LMn{Au^*}\cdot \LMn{A(s^k-x^*)} = -\delta  \sqrt{f(s^k)-f_*},\label{eq:dual}
\end{eqnarray}
where we have used \eqref{eq:id_f} and denoted $\delta := \sqrt 2 \LMn{Au^*}$. Our goal is to show that
$k (\hat F_k-g_*)$ is bounded from below. Indeed
\begin{eqnarray}
    k(\hat F_k - g_* ) &\geq& k(g(s^k)-g_*) + \sigma k^2 (f(s^k)-f_*) \nonumber\\
    &\geq& \frac{\sigma k^2}{2}(f(s^k)-f_*)+
    \left(\frac{\sigma k^2}{2}(f(s^k)-f_*) - \delta k\sqrt{f(s^k)-f_*}\right)
    %
. \label{eq:Fk}
\end{eqnarray}
The real-valued function $t \mapsto \frac{\sigma k^2}{2}t^2
-\delta k t$ attains its smallest value $-\frac{\delta ^2}{2\sigma}$ at
$t=\frac{\delta }{\sigma k}$. Hence, we can deduce in~\eqref{eq:Fk} that
\begin{equation}\label{eq:est_for_Fk}
    k(\hat F_k - g_*) \geq \frac{\sigma k^2}{2}(f(s^k)-f_*)-\frac{\delta ^2}{2\sigma}
    \geq -\frac{\delta^2}{2\sigma}.
\end{equation}
We can then apply the supermartingale theorem to the shifted
random variable $V_k(x^*)+ \delta^2/(2\sigma)$ to conclude that
\[
\sum_{k=0}^\infty \sigma (f(x^k)-f_*)+ \frac{p}{2} \LMn{\hat x^{k+1}-x^k}^2_\LMe<\infty
\quad a.s.
\]
and the sequence $V_k(x^*)$ converges almost surely to a random variable with
distribution bounded below by
$-\frac{\delta ^2}{2\sigma}$.  Thus $(f(x^k)-f_*)=o(1/k)$,
$\lim_{k\to\infty}\LMn{\LMhx^{k+1}-x^k}_\LMe^2 =0$
and, by the definition of $V_k(x^*)$, the sequence
$(x^k)$ is pointwise a.s. bounded thus $k(\hat F_k - g_*)$ is
bounded above and hence, by \eqref{eq:est_for_Fk}, $f(s^k)-f_* = O(1/k^2)$ a.s.

By the definition of $z^k$ in
step~3 of Algorithm~\ref{alg:tseng} with $\theta_k = \frac{1}{k+1}$, we
conclude that $f(z^k)-f_* = O(1/k^2)$ a.s.  This yields a useful
estimate:
\begin{multline}
    \label{eq:est_grad} \LMlr{\nabla f(z^k), \LMhx^{k+1}-x^*}
=\LMlr{A^T(Az^k-b), \LMhx^{k+1}-x^*} = \LMlr{A^TA(z^k-x^*), \LMhx^{k+1}-x^*}
\\ \leq \LMn{A(z^k-x^*)} \LMn{A(\LMhx^{k+1}-x^*)} =
\sqrt{(f(z^k)-f_*)(f(\LMhx^{k+1})-f_*)}.
\end{multline}
Since $f(\LMhx^{k+1})-f_* \to 0$ a.s., we have
$\frac{\sigma}{\theta_k} \LMlr{\nabla f(z^k), \LMhx^{k+1}-x^*}\to 0$.

Now consider the pointwise sequence of realizations of the random variables $x^k$
over $\LMo\in\Omega$ for which the sequence is bounded.  Since these
realizations of sequences are bounded, they possess cluster points.  Let
$x'$ be one such cluster point.  This point is feasible since
in the limit $f(x')-f_* = 0$.  Writing down inequality~\eqref{eq:2}
for the convergent subsequence with $x=x^*$ and taking the limit
we obtain $g(x') \leq g(x^*)$. Thus the pointwise cluster points
$x'\in S$ a.s.

We show next that the cluster points
are unique.  Again, for the same pointwise realization over $\LMo\in\Omega$
as above, suppose there is another cluster point $x''$.  By the argument
above, $x''\in S$.  The point $x^*\in S$ was arbitrary and so we can replace
this by $x'$.  To reduce clutter, denote $\alpha_k = k(\hat F_k
-g_*)$ and denote the subsequence of $(x^k)$ converging to $x'$ by
$\left(x^{k_i}\right)$ for $i\in\mathbb{N}$ and the subsequence converging to
$x''$
by $\left(x^{k_j}\right)$ for $j\in\mathbb{N}$. We thus have
\begin{align} \lim_{k\to \infty }V_k(x') & = \lim_{i\to
\infty}V_{k_i}(x') =
\lim_{i\to\infty}(\frac{p^2}{2}\LMn{x^{k_i}-x'}^2_{\tau^{-1}} +
\alpha_{k_i}) =\lim_{i\to\infty}\alpha_{k_i}\\ & =
\lim_{j\to\infty}V_{k_j}(x') = \lim_{j\to\infty}(\frac{p^2}{2}
\LMn{x^{k_j}-x'}^2_{\tau^{-1}} + \alpha_{k_j}) = \LMn{x''-\tilde
x_1}^2_{\tau^{-1}} + \lim_{j\to\infty}\alpha_{k_j},
\end{align}
from which $\lim_{i\to\infty}\alpha_{k_i} =
\frac{p^2}{2}\LMn{x''-x'}^2_{\tau^{-1}} +
\lim_{j\to\infty}a_{k_j}$ follows. Similarly, replacing $\tilde x$ with $x''$,
we derive $\lim_{j\to\infty}\alpha_{k_j} = \frac{p^2}{2}\LMn{\tilde
x_1-x''}^2_{\tau^{-1}} + \lim_{i\to\infty}\alpha_{k_i},$ from which
we conclude that $x' = x''$. Therefore, the whole
sequence $(x^k)$ converges pointwise almost surely to a solution.

\emph{Convergence of $(s^k)$}. First, recall from Lemma~\ref{l:sk}
that $s^k$ is a convex combination of $x^j$ for $j=0,\dots,
k$. Second, notice that for all $j$, $\beta_{k+1}^j \to 0$ as
$k\to \infty$. Hence, by the Toeplitz theorem (Exercise 66 in
\cite{polya}) we conclude that $(s^k)$ converges pointwise almost surely to the
same solution as $(x^k)$. By this, \eqref{th:main i} is proved.

\smallskip

\eqref{th:main ii}  Taking the total expectation $\LME$ of both sides
of \eqref{eq:10}, we obtain
\begin{equation*} \LME\bigl[V_{k+1}(x^*) +
\frac{p}{2}\LMn{\LMhx^{k+1}-x^k}^2_\LMe\bigr] \leq \LME[ V_k(x^*)].
\end{equation*}
Iterating the above,
\begin{equation}
    \label{eq:full0} \LME\Bigl[\frac{p^2}{2}\LMn{x^{k}-x^*}^2_{\tau^{-1}} +
k(\hat F_{k} - g_*) +
\sum_{i=1}^{k}\frac{p}{2}\LMn{\LMhx^{i}-x^{i-1}}^2_\LMe\Bigr] \leq
\frac{p^2}{2}\LMn{x^{0}-x^*}^2_{\tau^{-1}} =: C.
\end{equation} Since $g(s^k)\leq \hat g_k$, one has
\begin{equation}
    \label{eq:full} \LME\Bigl[\frac{p^2}{2}\LMn{x^{k}-x^*}^2_{\tau^{-1}} +
\sigma k^2(f(s^{k}) - f_*) +
\sum_{i=1}^{k}\frac{p}{2}\LMn{\LMhx^{i}-x^{i-1}}^2_\LMe\Bigr] \leq C +
\LME[k(g_* - g(s^k))].
\end{equation}
From the above equation it follows that
\begin{equation}
    \label{eq:full2} \LME\Bigl[\frac{p^2}{2k}\LMn{x^{k}-x^*}^2_{\tau^{-1}} +
\sigma k (f(s^k)-f_*)\Bigr] \leq \frac{C}{k}+ \LME[g_*-g(s^k)].
\end{equation} Recall that by our assumption, $g$ is bounded from
below: let $g(x)\geq -l$ for some $l>0$. Thus, from \eqref{eq:full2}
one can conclude that
\begin{equation*} \LME[f(s^k)- f_*]\leq \frac{C}{\sigma k^2} +
  \frac{g_*+l}{\sigma k}.
\end{equation*}
This means that almost surely $f(s^k)-f_* =
O(1/k)$. Later we will improve this estimation. Let $\Omega_1$ be an
event of probability one such that for every $\LMo\in \Omega_1$,
$f(s^k(\LMo ))-f_* = O(1/k)$.

\emph{Boundedness of $(s^k)$}. First, we show that $(s^i(\LMo ))_{i\in
\mathcal{I}}$ with $\mathcal{I} = \{i\colon g(s^i) <g_*\}$ is a
bounded sequence for every $\LMo\in \Omega_1$.  For this we revisit the
arguments from \cite{solodov2007explicit,malitsky2017}.  By assumption
the set $S = \{x\colon g(x)\leq g_*, f(x)\leq f_*\}$ is nonempty and
bounded. Consider the convex function $\varphi(x) = \max\{g(x)- g_*,
f(x)-f_* \}$. Notice that $S$ coincides with the level set
$\mathcal{L}(0)$ of $\varphi$:
\begin{equation*} S = \mathcal{L}(0) = \{ x\colon \varphi(x)\leq 0\}.
\end{equation*} Since $\mathcal{L}(0)$ is bounded, the set
$\mathcal{L}(c)=\{x\colon \varphi(x)\leq c\}$ is also bounded for any $c\in
\LMR$. Fix any $c\geq 0$ such that $f(s^{k}(\LMo ))\leq c$ for all $k$ and
$\LMo\in \Omega_1$. Since $g(s^{i})-g_*<0\leq c$ for $i \in
\mathcal{I}$, we have that for every $\LMo\in \Omega_1$, $s^{i}(\LMo )\in
\mathcal{L}(c)$, which is a bounded set. Hence, $(s^{i}(\LMo ))_{i\in
\mathcal I}$ is bounded for every $\LMo\in \Omega_1$.

Note that from \eqref{eq:full} we have a useful consequence: both
sequences $(x^j)$, $(\LMhx^{j}-x^{j-1})$ with $j \notin \mathcal{I}$ are
pointwise a.s.\ bounded.  From this it follows immediately that for $j\notin
\mathcal I$, $(x^j-x^{j-1})$ is pointwise a.s.\ bounded as well. Hence, there
exists an event $\Omega_2$ of probability one such that $x^j(\LMo )$,
$(x^{j}-x^{j-1})(\LMo )$ are bounded for each $\LMo\in \Omega_2$.
We now restrict ourselves to points $\LMo\in \Omega_1 \cap \Omega_2$.
Clearly, the latter set has a probability one.
Let $M_\LMo$
be such a constant that, for $\LMo\in \Omega_1 \cap \Omega_2$
\begin{equation}
    \begin{aligned} &\LMn{s^i(\LMo )}\leq M_\LMo \quad & \forall i\in \mathcal I
\\ & \LMn{(x^j+ (n-1)(x^j-x^{j-1}))(\LMo )}\leq M_\LMo \quad & \forall j\notin
\mathcal I.
\end{aligned}
\end{equation}
We prove by induction that the sequence
$(s^k(\LMo))$ is bounded, hence $(s^k)$ is pointwise a.s. \ bounded. Suppose that
for index $k$, $\LMn{s^k(\LMo )} \leq M_\LMo$. If for index $k+1$,
$g(s^{k+1}(\LMo ))< g_*$ then we are done: $k+1\in \mathcal{I}$ and
hence, $\LMn{s^{k+1}(\LMo )}\leq M_\LMo$. If $g(s^{k+1}(\LMo ))\geq g_*$, then
$(k+1)\notin \mathcal I$. By the definition of $s^k$ in
Algorithm~\ref{alg:tseng}, $s^{k+1} =
\theta_k(nx^{k+1}-(n-1)x^k)+(1-\theta_k)s^k$, and hence
\begin{equation*} \LMn{s^{k+1}} \leq \theta_k \LMn{x^{k+1}+ (n-1)(x^{k+1}-x^{k})} +
(1-\theta_k)\LMn{s^k},
\end{equation*} which shows that $\LMn{s^{k+1}(\LMo )}\leq M_\LMo$. This
completes the proof that $s^k(\LMo )$ is bounded and hence $(s^k)$ is
pointwise a.s.\ bounded.

\emph{Convergence of $(s^k)$.} Recall that a.s.\
$f(s^k)-f_* = O(1/k)$. Hence, all limit points of $s^k(\LMo )$ are
feasible for any $\LMo\in \Omega_1\cap \Omega_2$. This means that for
every $\LMo\in \Omega_1\cap \Omega_2$,
$\liminf_{k\to\infty} g(s^k(\LMo ))\geq g_*$. If one assumes that
there exist an event $\Omega_3$ of non-zero probability such that for
every $\LMo\in \bigcap_{i=1}^3\Omega_i$,
$\limsup_{k\to\infty} g(s^k(\LMo )) > g_*$, then taking the limit
superior in \eqref{eq:full2}, one obtains a contradiction. This yields
that almost surely $\lim_{k\to \infty} g(s^k) = g_*$ and hence, almost
surely all limit points of $(s^k)$ belong to $S$. Using the obtained
result, we can improve the estimation for $f(s^k)-f_*$. Now from
\eqref{eq:full2} it follows that almost surely $f(s^k)-f_* = o(1/k)$.
\qed
\end{proof}

\subsection{Possible generalizations}
\label{sec:generalizations}

\textbf{Strong convexity.} When $g$ is strongly convex in
\eqref{eq:main}, the estimates in Theorem~\ref{th:main} can be
improved and in case (ii) the convergence of $(s^k)$ to the solution
of \eqref{eq:main} can be proved. For this one needs to combine the
proposed analysis and the analysis from~\cite{malitsky2017}.

\smallskip
\noindent \textbf{Adding the smooth term and ESO}. Instead, of solving
\eqref{eq:main}, one can impose more structure for this problem:
\begin{equation}
    \label{eq:gen1}
    \min_x g(x) + h(x) \quad \text{s.t.}\quad Ax=b,
\end{equation}
where in addition to the assumptions in Section~\ref{sec:intro}, we
assume that $h\colon \LMR^N\to \LMR$ is a smooth convex function with
a Lipschitz-continuous gradient $\nabla h$. In this case the new
coordinate algorithm can be obtained by combining analysis from
\cite{fercoq2015accelerated} and \cite{malitsky2017}. The expected
separable overapproximation (ESO), proposed in \cite{fercoq2015accelerated},
can help here much more, due to the new additional term $h$.

\smallskip
\noindent
\textbf{Parallelization.} For simplicity, our analysis has focused
only on the update of one block in every iteration. Alternatively,  one could
choose a random subset of blocks and update all of them more or less
by the same logic as in Algorithm~\ref{alg:pd}. This is one of the
most important aspects of the algorithm that provides for fast implementations. The analysis of
such extension will involve only a bit more tedious expressions with
expectation, for more details we refer again to
\cite{fercoq2015accelerated}, where this was done for the
unconstrained case: $\min_x g(x)+h(x)$. More challenging is to develop
an asynchronous parallel extension of the proposed method in the
spirit of \cite{peng2016arock}, but with the possibility of taking larger
steps.

\section{Numerical experiments}
\label{sec:numer-exper}
This section collects several numerical tests to illustrate the
performance of the proposed methods. Computations\footnote{All codes
can be found on \url{https://gitlab.gwdg.de/malitskyi/coo-pd.git}}
were performed using Python 3.5 on a standard laptop running 64-bit Debian GNU/Linux.

\subsection{Basis pursuit}
For a given  matrix $A\in \LMR^{m\times n}$ and observation vector $b\in
\LMR^m$, we want to find a sparse vector $x^\dag \in \LMR^n$ such that
$Ax^\dag = b$. In standard circumstances we have $m\ll n$, so the linear
system is underdetermined. Basis pursuit, proposed by Chen and Donoho
\cite{chen:basis_pursuit,chen2001atomic},  involves  solving
\begin{equation}
    \label{num:bp-1}
    \min_{x} \LMn{x}_1 \quad \text{s.t.}\quad Ax = b.
\end{equation}
In many cases it can be shown (see \cite{candes2008introduction} and
references therein) that  basis pursuit is able to recover
the true signal $x^\dag$  even when $m\ll n$.

A simple observation says that problem~\eqref{num:bp-1} fits well
in our framework. For given $m,n$ we generate the matrix $A\in \LMR^{m\times n}$ in two ways:
\begin{itemize}
    \item Experiment 1: $A$ is a random Gaussian matrix, i.e.\ every
    entry of $A$ drawn from the normal distribution
    $\mathcal{N}(0,1)$. A sparse vector $x^\dag\in \LMR^n$ is
    constructed by choosing at random $5\%$ of its entries
    independently and uniformly from $(-10,10)$.

    \item Experiment 2: $A = R\Phi$, where $\Phi \in \LMR^{n\times n}$ is
    the discrete cosine transform and $R\in \LMR^{m\times n}$ is the
    random projection matrix: $Rx$ randomly extracts $m$ coordinates from
    the given $x$. A sparse vector $x^\dag\in \LMR^n$ is
    constructed by choosing at random $50$ of its entries among the first
    $100$ coordinates independently from the normal distribution $\mathcal N(0,1)$. The
    remaining entries of $x^\dag$ are set to $0$.
\end{itemize}

For both experiments we generate $b = Ax^\dag$. The starting point
for all algorithms is $x^0=0$.  For all methods we use the same
stopping criteria:
\begin{equation*}
\LMn{Ax^k-b}_\infty \leq 10^{-6}\quad \text{and}\quad \deltaist(-A^Ty^k,
\partial_{\LMn{\cdot}_1}(x^k))_\infty\leq 10^{-6}.
\end{equation*}

For given  $A$ and $b$ we compare the performance of
PDA and the proposed methods: Block-PDA with $n_{\text{block}} = n/w$ blocks of the
width $w$ each and Coo-PDA with every single coordinate as a
block. In the first and the second experiments the width of blocks for
the Block-PDA is $w=50$.

The numerical behavior of all three methods depends strongly on the choice
of the stepsizes. It is easy to artificially handicap PDA in comparisons with 
almost any algorithm: just take very bad steps for PDA.  To be
fair, for every test problem we present the best results of PDA among
all with steps $\sigma = \dfrac{1}{2^j \LMn{A}}$, $\tau = \dfrac{2^j}{\LMn{A}}$, for
$j=-15,-14,\dots, 15$. Instead, for the proposed methods we always take
the same step $\sigma = \dfrac{1}{2^j n_{\text{block}}}$, where we set
$j=11$ for the first experiment and $j=8$ for the second one.


\begin{table}\centering
    \footnotesize
    \caption{Comparison of PDA, Block-PDA and Coo-PDA for
    problem~\eqref{num:bp-1}. Experiment 1}
    \label{tab:bp-1}
\ra{1}
\begin{tabular}{@{}rrrrrrrrrr@{}}\toprule
Algorithm&  \phantom{b}& \multicolumn{2}{c}{$m=1000$, $n = 4000$} & \phantom{abc}& \multicolumn{2}{c}{$m=2000$, $n=8000$} &
\phantom{abc} & \multicolumn{2}{c}{$m=4000$, $n=16000$}\\
  \cmidrule{3-4} \cmidrule{6-7}  \cmidrule{9-10}
         & &  epoch &  CPU && epoch &  CPU && epoch &  CPU \\
  \midrule \\
  \addlinespace[-0.3cm]
  PDA       && 777 & 24 && 815 & 89 && 829 & 333\\
  Block-PDA && 108 & 4 && 103 & 12 && 107 & 51\\
  Coo-PDA   && 79  & 2 && 73  & 7 && 94 & 34\\
\bottomrule
\end{tabular}
\end{table}

Tables~\ref{tab:bp-1} and \ref{tab:bp-2} collect information of how
many epochs and how much CPU time (in seconds) is needed for each
method and every test problem to reach the $10^{-6}$ accuracy. The
term ``epoch'' means one iteration of the PDA or $n_{\text{block}}$
iterations of the coordinate PDA. By this, after $k$ epochs the $i$-th
coordinate of $x$ will be updated on average the same number of times
for the PDA and our proposed method. The CPU time might be not a good
indicator, since it depends on the platform and the
implementation. However, the number of epochs gives an exact number of
arithmetic operations. This is a fair benchmark characteristic.

Concerning the stepsizes, we reiterate that the PDA was always taken with
the best steps for every particular problem: for the first experiment
(when $A$ is Gaussian), the parameter $j$ was usually among $4,5,6,7$.
For the second experiment, the best values of $j$
were among $0,\dots, 6$.


\begin{table}\centering
    \footnotesize
        \caption{Comparison of PDA, Block-PDA and Coo-PDA for
    problem~\eqref{num:bp-1}. Experiment 2}
    \label{tab:bp-2}
\ra{1}
\begin{tabular}{@{}rrrrrrrrrr@{}}\toprule
Algorithm&  \phantom{b}& \multicolumn{2}{c}{$m=1000$, $n = 4000$} & \phantom{abc}& \multicolumn{2}{c}{$m=2000$, $n=8000$} &
\phantom{abc} & \multicolumn{2}{c}{$m=4000$, $n=16000$}\\
  \cmidrule{3-4} \cmidrule{6-7}  \cmidrule{9-10}
  & &  epoch &  CPU && epoch &  CPU && epoch &  CPU \\
  \midrule \\
  \addlinespace[-0.3cm]
  PDA && 303 & 9 && 284 & 29 && 286 & 122\\
  Block-PDA && 41 & 2 && 40 & 5 && 36 & 18\\
  Coo-PDA && 27 & 1 && 23 & 3 && 24 & 10\\
\bottomrule
\end{tabular}
\end{table}

\subsection{Basis pursuit with noise}
In many cases, a more realistic scenario is when our measurements are
noisy. Now we wish to recover a sparse vector $x^\dag \in \LMR^n$ such
that $Ax^\dag + \LMe = b $, where $\LMe \in \LMR^m$ is the unknown
noise. Two standard approaches \cite{chen2001atomic,candes2006stable}
involve solving either the lasso problem or basis pursuit
denoising:
\begin{equation}
    \label{num:bpd} \min_x \LMn{x}_1 + \frac{\delta }{2}\LMn{Ax-b}^2 \qquad
\text{or}\qquad \min_x \LMn{x}_1 \quad \text{s.t.}\quad \LMn{Ax-b}\leq \delta .
\end{equation}
In order to recover $x^\dag $, both problems require a delicate tuning
of the parameter $\delta $. For this one usually requires   some a priori knowledge about
the noise.

Even with noise, we still apply our method to the
plain basis pursuit problem
\begin{equation}
    \label{num:bp-2}
    \min_x \LMn{x}_1\quad\text{s.t.}\quad Ax=b.
\end{equation}
We know that the proposed method converges even when the linear system
is inconsistent (and this is the case when $A$ is
rank-deficient). First, we may hope that the actual solution $x^*$ of
\eqref{num:bp-2} to which the method converges is not far from the
true $x^\dag$. In fact, if the system $Ax=b$ is consistent, we have
$\LMn{A(x^*-x^\dag)}=\LMe$. Similarly, if that system  is inconsistent, by \eqref{eq:id_f}
we have $f(x^\dag)-f(x^*)=\frac 1 2 \LMn{A(x^\dag-x^*)}^2$. And hence,
again we obtain $ \LMn{A(x^\dag-x^*)}^2\leq 2f(x^\dag)=\LMe^2 $.
Second, it might happen that the trajectory of $(x^k)$ at some point
is even closer to $x^\dag$; in this case the early stopping
can help to recover $x^\dag$.  In fact, as our simulations show, the
convergence of the coordinate PDA to the actual solution $x^*$ is very
slow, though, it is unusually fast in the beginning.

For given $m=1000,n=4000$ we generate two scenarios:
\begin{itemize}
    \item $A\in \LMR^{m\times n}$ is a random Gaussian matrix.
    \item $A = A_L A_R$, where $A_L \in \LMR^{m\times m/2}$ and $A_R \in
    \LMR^{m/2 \times n}$ are two random Gaussian matrices.
\end{itemize}
The latter condition guarantees that the rank of $A$ is at most $m/2$. Hence, with a high
probability the system $Ax=b$ will be inconsistent due to the noise.

We generate a sparse vector $x^\dag \in \LMR^n$ choosing at random $50$ of
its elements  independently and uniformly from $(-10,10)$. Then we
define $b\in \LMR^m$ in two ways: either $b = Ax^\dag+\LMe$, where $\LMe\in \LMR^m$
is a random Gaussian vector or $b$ is obtained by rounding off
every coordinate of $Ax^\dag$ to the nearest integer.

For simplicity we run only the block PDA with blocks of width $50$,
thus we have $n_{\text{block}} = n/50$. The parameter $\sigma$ is chosen
as $\sigma = \frac{1}{2^{25}n_{\text{block}}}$. After every $k$-th epoch we
compute the signal error $\frac{\LMn{x^k-x^\dag}}{\LMn{x^\dag}}$ and the
feasibility gap $\LMn{A^T(Ax^k-b)}$. The results are presented in
Fig.~\ref{fig:gauss_matrix}, \ref{fig:lowrank_matrix}.

As one can see from the plots, the convergence of $(x^k)$ to the best
approximation of $x^\dag$ takes place just after few epochs (less than
100). This convergence is very fast  for both the signal error and the
feasibility gap. After that, both lines switch to the slow regime:
the signal error slightly increases and stabilizes after that; the
feasibility gap decreases relatively slowly. Although in practice we do
not know $x^\dag$, we still can use early stopping, since both lines
change their behavior approximately at the same time and it is easy
to control the feasibility gap.

\begin{figure}[ht]
    \centering
    \includegraphics[width=\linewidth]{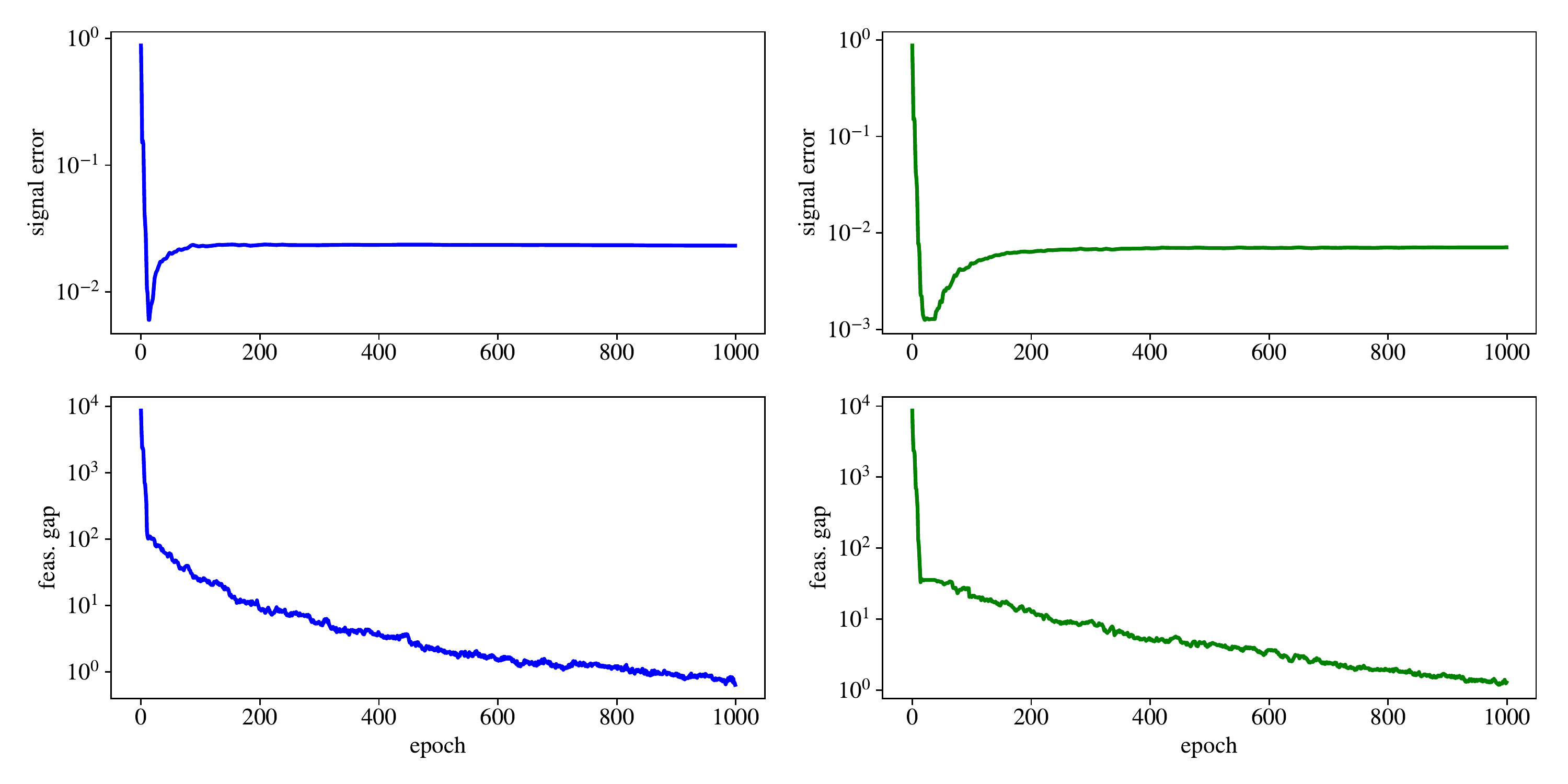}
    \caption{Convergence of block-PDA for noisy basis pursuit. $A$ is
    a random Gaussian matrix. Left (blue): problem with
    $b=Ax^\dag+\LMe$, where $\LMe$ is a random Gaussian vector. Right
    (green): problem with $b$ obtained by rounding off $Ax^\dag$ to
    the nearest integer. Top: signal error. Bottom: feasibility
    gap.}
    \label{fig:gauss_matrix}
\end{figure}

\begin{figure}[ht]
    \centering
    \includegraphics[width=\linewidth]{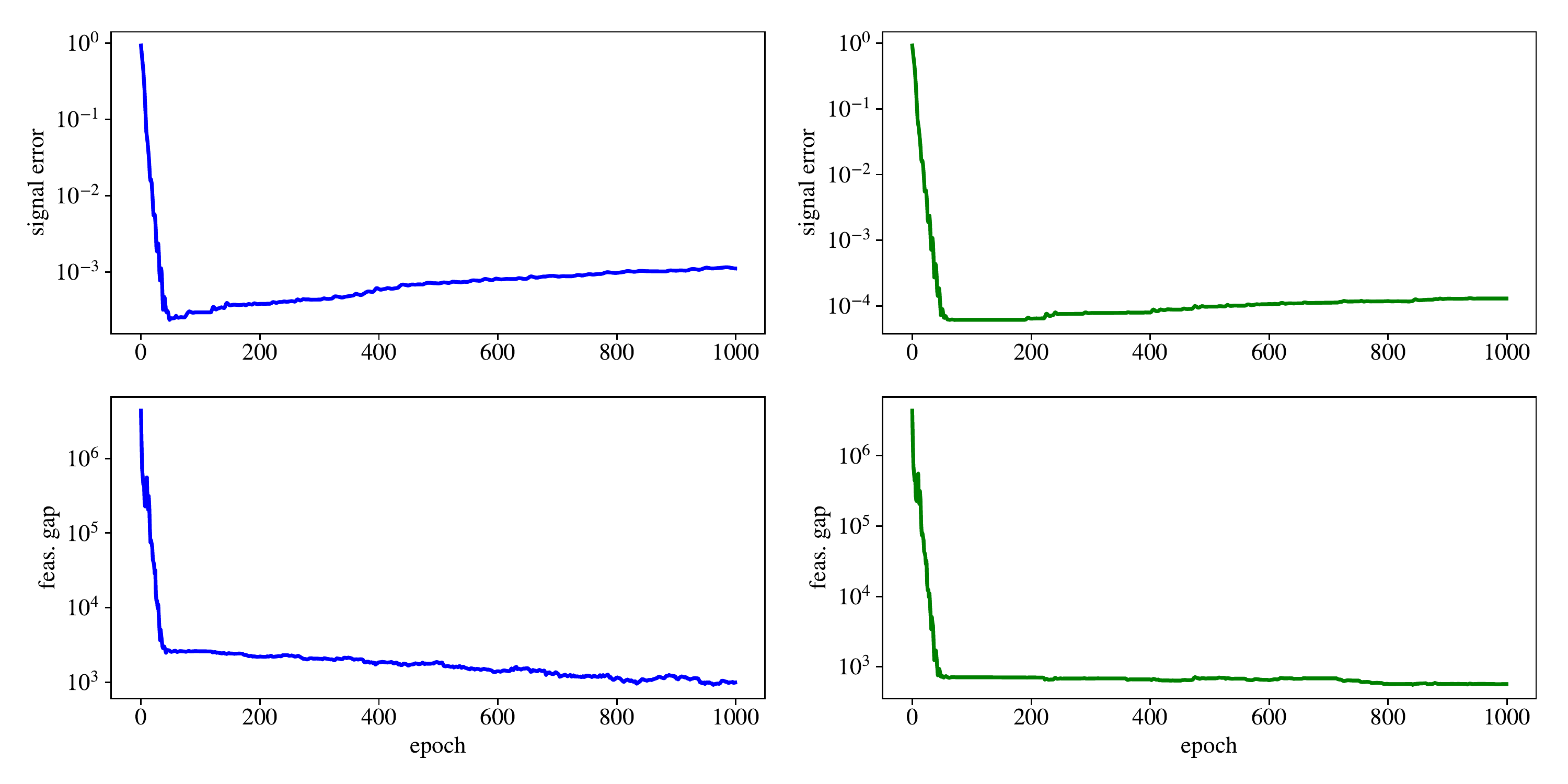}
    \caption{Convergence of block-PDA for noisy basis pursuit. $A$ is
    a low-rank matrix $A=A_LA_R$. Left (blue): problem with
    $b=Ax^\dag+\LMe$, where $\LMe$ is a random Gaussian vector. Right
    (green): problem with $b$ obtained by rounding off $Ax^\dag$ to
    the nearest integer. Top: signal error. Bottom: feasibility
    gap.}

    \label{fig:lowrank_matrix}
\end{figure}
Finally, we would like to illustrate the proposed approach on a
realistic non-sparse signal. We set $m=1000, n=4000$ and generate a random signal
$w\in \LMR^n$ that has a sparse representation in the dictionary of the
discrete cosine transform, that is $w = \Phi x^\dag$ with the matrix
$\Phi\in \LMR^{n\times n}$ is the discrete cosine transform and
$x^\dag \in \LMR^n$ is the sparse vector with only $50$ non-zero coordinates
drawn from $\mathcal N(0,1)$.  The measurements are modeled by a random
 Gaussian matrix $M\in \LMR^{m\times n}$. The observed data is
corrupted by noise: $b = Mw + \LMe$, where $\LMe\in \LMR^m$ is a
random vector, whose entries are drawn from $\mathcal N(0,10)$.

Obviously, we can rewrite the above equation as $b = Ax + \LMe$, where
$A = M\Phi\in \LMR^{m\times n}$.  We apply the proposed block-coordinate
primal-dual method to the problem
\begin{equation}
    \label{num:denois3}
    \min_x \LMn{x}_1 \quad \text{s.t.}\quad A x = b
\end{equation}
with $\sigma = \frac{1}{2^{22}n_{\text{block}}}$ and the width $50$ of each block. The behavior of the
iterates $x^k$ is depicted in Figure~\ref{fig:real_sign}.
Figure~\ref{fig:real_sign2} shows the true signal $w$ and the reconstructed signal $\hat w =
\Phi x^{30}$ after $30$ epochs of the method. The signal error in this
case is $\frac{\LMn{\hat w -w}}{\LMn{w}} = 0.0036$. Interestingly, after $1000$
epochs we obtain the reconstructed signal $\hat w^{1000}$ for which the
error is also quite reasonable: $\frac{\LMn{\hat w^{1000} -w}}{\LMn{w}} = 0.016$.
\begin{figure}[ht]
    \centering
    \includegraphics[width=\linewidth]{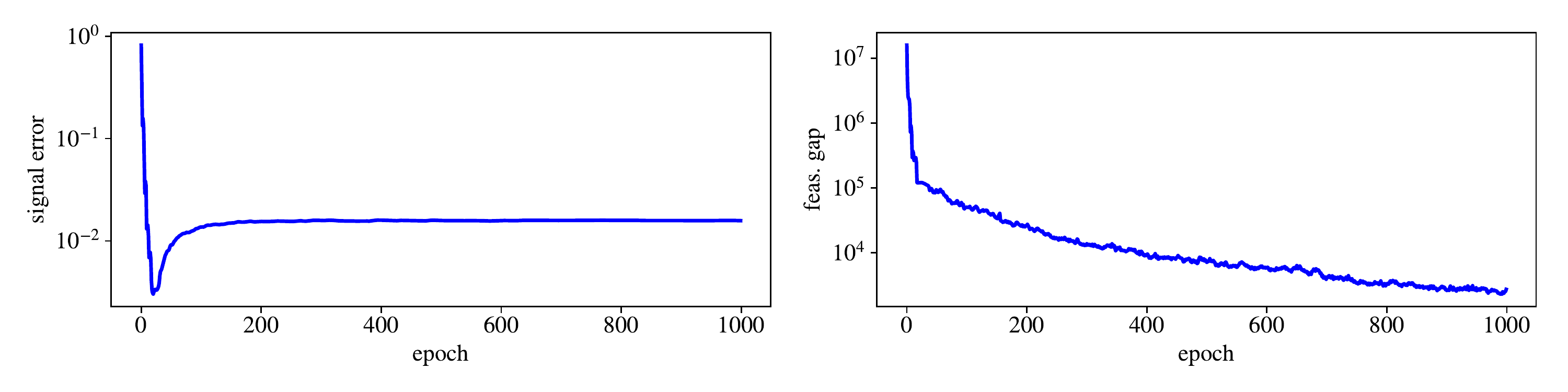}
    \caption{Convergence of block-PDA for noisy basis
    pursuit with a non-sparse signal. $A=M\Phi$. Left: signal error. Right: feasibility
    gap.}
    \label{fig:real_sign}
\end{figure}

\begin{figure}[ht]
    \centering
    \includegraphics[width=\linewidth]{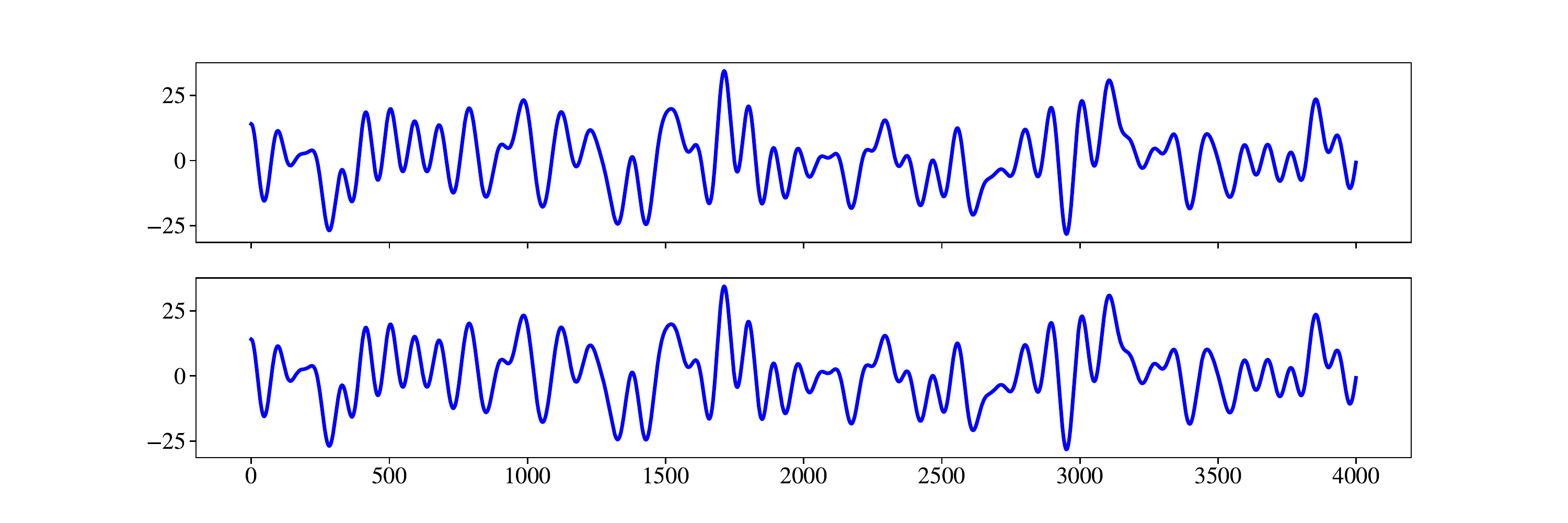}
    \caption{Top: original signal $w$. Bottom: reconstructed signal
    $\hat w$ after $30$ epochs. The signal error is $0.0036$.
    }
    \label{fig:real_sign2}
\end{figure}

Notice that we have reconstructed the signal without any knowledge about
the noise. The only parameter which requires some tuning from our side was the
stepsize $\sigma$. However, the method is not very sensitive to the choice of $\sigma$. In
fact, we were able to reconstruct the signal at least for any $\sigma
=\frac{1}{2^{j}n_{\text{block}}}$ for $j$ from the range $15,\dots, 30$.
Only the number of iterations of the method and the obtained accuracy
changed, though still it was always enough for a good reconstruction.

\subsection{Robust principal component analysis}
\label{sec:rpca}

Given the observed data matrix $M\in \LMR^{n_1\times n_2}$, robust
principal component analysis aims to find a low rank matrix $L$ and a
sparse matrix $S$ such that $M = L+S$. Problems which
involve so many unknowns and the ``low-rank'' component should be
difficult (in fact it is NP-hard), however, they can be
successfully handled via robust principal component analysis, modeled by the
following convex optimization problem
\begin{equation}
    \label{rpca:main} \min_{L, S} \LMn{L}_* + \lambda \LMn{S}_1\quad
\text{s.t.}\quad L + S = M.
\end{equation} 
Here $\LMn{L}_* = \sum_i \sigma_i(L)$ is the nuclear norm of
$L$, $\LMn{S}_1 = \sum_{i,j}|S_{ij}|$, and $\lambda >0$. In
\cite{candes2011robust,wright2009robust} it was shown that under some mild assumptions
problem~\eqref{rpca:main} indeed recovers the true matrices $L$ and
$S$. The RPCA has many applications, see \cite{candes2011robust} and
the references therein.

Problem~\eqref{rpca:main} is well-suited for both
ADMM~\cite{yuan2009sparse,lin2010augmented} and PDA. In fact, since
the linear operator $A$ that defines the linear constraint has a simple
structure $A = [\begin{array}{c|c} I & I \end{array}]$, those two
methods almost coincide.  Obviously, the bottleneck for both methods
is computing the prox-operator with respect to the variable $L$, which involves
computing a singular value decomposition (SVD). As \eqref{rpca:main}
is a particular case of \eqref{eq:main} with two blocks, one can apply
the proposed coordinate PDA. In this case, in every iteration one
should update either the $L$ block or the $S$ block. Hence, on average $k$
iterations of the method require only $k/2$ SVD. Keeping this in mind,
we can hope for a faster convergence of our method compared to the
original PDA.

For the experiment we use the settings from \cite{wright2009robust}. For given
$n_1,n_2$ and rank $r$ we set $L^\dag = Q_1 Q_2$ with
$Q_1 \in \LMR^{n_1 \times r}$ and $Q_2 \in \LMR^{r\times n_2}$, whose
elements are i.i.d.\ $\mathcal N (0,1)$ random variables.
Then we generate $S^\dag\in \LMR^{n_1\times n_2}$  as a sparse matrix whose $5\%$ non-zero
elements are chosen uniformly at random from the range $[-500,
500]$. We set $M = L^\dag + S^\dag$ and $\lambda = \frac{1}{\sqrt{n_1}}$.

The KKT optimality conditions for \eqref{rpca:main} yield
\begin{equation*}
    Y \in \partial (\lambda \LMn{S}_1), \qquad Y \in \partial
    (\LMn{L}_*), \qquad S + L - M = 0,
\end{equation*}
which implies that $(L,S)$ is
optimal whenever
\begin{equation*}
    \partial
    (\lambda \LMn{S}_1) \cap \partial (\LMn{L}_*) \neq \varnothing \quad
    \text{and} \quad S + L -M = 0.
\end{equation*}
Since we do not want to compute $\partial (\LMn{L}_*)$ and
$\partial (\LMn{S}_1)$, we just measure the distance between their two
subgradients which we can compute from the iterates.

Given the current iterates $(L^k, S^k, Y^k)$ for the PDA,
we know that $\frac{1}{\tau} (L^{k-1}-L^k) - Y^{k-1}\in \partial
(\LMn{L^k}_*)$ and $\frac{1}{\tau} (S^{k-1}-S^k) - Y^{k-1}\in \partial
(\LMn{S^k}_1)$. Hence, we can terminate algorithm whenever
\begin{equation*}
    \frac{\LMn{L^{k-1}-L^k+S^{k}-S^{k-1}}}{\tau \LMn{M}} \leq \LMe \quad \text{and} \quad  \frac{\LMn{S^k+L^k-M}_{\infty}}{\LMn{M}} \leq \LMe
\end{equation*}
where $\LMe>0$ is the desired accuracy. A similar criteria was used for the
Coo-PDA termination.

We run several instances of PDA with stepsize $\tau = \frac{2^j}{L}$,
$\sigma = \frac{1}{2^j L}$ for $j=-6,-5,\dots, 12$, where $L=\sqrt 2$ is
the norm of the operator $A= [\begin{array}{c|c} I &
  I \end{array}]$. Similarly, we run several instances of Coo-PDA with
$\tau = (1,1)$ and $\sigma = \frac{1}{2^j}$ for the same range of indices
$j$. We show only the performance of the best instances for both
methods, and it was always the case that $j\in \{6,7,8\}$. The
accuracy $\LMe$ was set to $10^{-6}$.  We also compare PDA-r and
Coo-PDA-r, where instead of an exact SVD a fast randomized SVD solver
was used. We chose the one from the scikit-learn library
\cite{scikit-learn}.

In Table~\ref{tab:rpca} we show the benchmark for PDA and Coo-PDA
with an exact evaluation of SVD and for PDA-r and Coo-PDA-r with
randomized SVD. For different input data, the table collects the total
number of iterations and the CPU time in seconds for all
methods. Notice that for the Coo-PDA the number of evaluation of SVD
is approximately half of the iterations, this is why it terminates
faster.

We want to highlight that in all experiments the unknown matrix
$L^\dag$ was indeed recovered. And, as one can see from the
Table~\ref{tab:rpca}, in all experiments the coordinate primal-dual
algorithms performed better than the standard ones.



\begin{table}\centering
    \footnotesize
        \caption{Comparison of PDA and Coo-PDA for
    problem~\eqref{rpca:main} with exact and approximate evaluation of
    SVD}
    \label{tab:rpca}
\ra{1.0}
\begin{tabular}{@{}rrrrrrrrrrrrrrr@{}}\toprule
$n_1$ & $n_2$ & $r$ & \phantom{b}& \multicolumn{2}{c}{PDA} & \phantom{abc}& \multicolumn{2}{c}{Coo-PDA} &
\phantom{abc} & \multicolumn{2}{c}{PDA-r} & \phantom{abc} & \multicolumn{2}{c}{Coo-PDA-r}\\
  \cmidrule{5-6}  \cmidrule{8-9}  \cmidrule{11-12} \cmidrule{14-15}
      & &  && iter &  CPU && iter &  CPU && iter &  CPU && iter &  CPU\\
  \midrule\\
  1000 & 500 & 20 && 161& 173&& 219 & 121 &&104 &39 &&159 &29\\
  1500 & 500 & 20 && 150& 211&& 200 & 142 &&139 &71&& 182 &47 \\
  2000 & 500 & 50 && 154& 279&& 205 & 185 &&130 &87 &&180 &61 \\
  1000 & 1000 &50 && 188& 1004&&  251 & 678 &&124 &111 &&174 & 76 \\
  2000 & 1000 &50 && 160& 1355&&  200 & 879 &&91 &146 &&120 &95 \\
\bottomrule
\end{tabular}
\end{table}

\begin{acknowledgement}
This research was supported by the German Research Foundation grant SFB755-A4.
\end{acknowledgement}
%




\end{document}